\documentclass[12pt]{article}
\usepackage{fullpage}

\usepackage{hyperref,amsthm,amssymb,amsmath,verbatim}

\hypersetup{
pdfborderstyle={/S/U/W 0.5},%
linkbordercolor        ={0 0 0},%         color of frame around internal links (if colorlinks=false)
citebordercolor        ={0 0 0},%        color of frame around citations
urlbordercolor         ={0 0 0}%         color of frame around URL links
}

\usepackage{lmodern}
\usepackage[T1]{fontenc}

\usepackage{booktabs}
\usepackage{tabularx}
\usepackage{cite}
\usepackage[small,bf]{caption}

\newcommand{\R}{{\mathbb R}}
\newcommand{\C}{\mathbb{C}}
\newcommand{\N}{\mathbb{N}}

\newcommand{\whg}{{\widehat{G}}}
\newcommand{\be}{\begin{equation}}
\newcommand{\ee}{\end{equation}}

\newtheorem{prop}{Proposition}
\newtheorem{lemma}{Lemma}
\newtheorem{defin}{Definition}
\newtheorem{theor}{Theorem}
\newtheorem{corol}{Corollary}

\allowdisplaybreaks

\begin{document}

\title{Examples of inconsistency in optimization by expected improvement}
\author{D. Yarotsky\footnote{Datadvance llc, Moscow,
\href{mailto:yarotsky@datadvance.net}{\nolinkurl{yarotsky@datadvance.net}}}
\footnote{Institute for Information Transmission Problems, Moscow,
\href{mailto:yarotsky@iitp.ru}{\nolinkurl{yarotsky@iitp.ru}}
}
}
\date{March 31, 2012}
\maketitle

\begin{abstract}
We consider the 1D Expected Improvement optimization based on Gaussian processes having spectral densities converging to zero faster than exponentially. We give examples of problems where the optimization trajectory is not dense in the design space. In particular, we prove that for Gaussian kernels there exist smooth objective functions for which the optimization does not converge on the optimum.

\end{abstract}

\tableofcontents

\section{Introduction}
\paragraph{The optimization problem.}
Consider a global ``black-box'' optimization problem
\begin{equation} \label{eq:mainopt}
f(x) \longrightarrow \min_{x\in \mathcal{D}}
\end{equation}
For the moment, suppose that $\mathcal{D}$ is a compact metric space, and $f$ a continuous real-valued  function on $\mathcal{D}$, so that  the minimum $f^* = \min_{x\in \mathcal{D}}f(x)$ exists.  
Consider an optimization procedure seeking this minimum. 
In black-box optimization, such a procedure consists of a sequence of iterations; each iteration suggests for evaluation a new point of the set $\mathcal D$ based on the already observed values of the  objective function $f$. 
More precisely, we can say that, algorithmically, optimization is defined by the initial point $x_1\in\mathcal{D}$ and a family of mappings 
\[\mathcal A_K: (\mathcal D\times\R)^K\to\mathcal D, \quad K=1,2,\ldots.\] The optimization trajectory $\{x_K\}_{K=1}^\infty$ is then determined by relations
\begin{equation}\label{eq:history}
x_{K+1} = \mathcal{A}_K\Big(\{x_k, f(x_k)\}_{k=1}^K\Big),\quad K=1,2,\ldots.\end{equation}
Any practical optimization is terminated at some step $K$, and the approximate minimum $f^*_K$ is then defined by \[f^*_K := \min_{k=1,\ldots,K}f(x_k).\]
It is then natural to call optimization \emph{consistent} if  
\[ \lim_{K\to\infty}f^*_K = f^*.\]

The following proposition is a very simple but important criterion of consistency on the space of continuous functions \cite{Torn:1989:GO:75021}.
\begin{prop}\label{thm:prop1}
An optimization algorithm defined by mappings $\mathcal A_K$ is consistent for all $f\in C(\mathcal D)$ if and only if for any continuous $f$ the trajectory $\{x_K\}_{K=1}^\infty$ generated by  \eqref{eq:history} is dense in $\mathcal D$.
\end{prop}
The sufficiency is clear; the necessity follows since any continuous function can be modified, preserving its continuity, in any open set so as to make the function attain its optimum in this open set.  

In many  practical applications, the objective function $f$ is expensive to evaluate, and the mappings $\mathcal A$ can then be quite complex and resource--intensive; in particular they often involve solving  auxiliary optimization problems. A popular modern approach to global black-box optimization is stochastic Bayesian optimization where these auxiliary problems are stated using some prior assumptions of probabilistic nature. In this paper we will consider one of the most natural and well-known methods of this type --  optimization by Expected Improvement \cite{Mockus_Tiesis_Zilinskas_1978, 
mockus1989bayesian,
schonlau1,
Schonlau:1997:CEG:926099,
JoScWe97a,
Jones:1998:EGO:596070.596218,
Jones:2001:TGO:596097.596412}

\paragraph{The Expected Improvement algorithm (EI).}   In this method, we think of the optimized function $f$ as a realization of a stochastic process $(\xi_x)_{x\in\mathcal D}$. Assuming the probability measure associated with the process is known, we define the mappings $\mathcal A_K$ by maximizing the expectation of the improvement in the best known value of the objective function resulting from its additional evaluation,  conditioned on the set $\{\xi_{x_k} = f(x_k)\}_{k=1}^K$. Precisely, we define
\begin{equation}\label{eq:ai}
\mathcal{A}_K\Big(\{x_k, f(x_k)\}_{k=1}^K\Big)
= \arg \max_{x\in \mathcal D} I_{K;\{x_k, f(x_k)\}_{k=1}^K} (x),
\end{equation}
where
\[
I_{K;\{x_k, f(x_k)\}_{k=1}^K} (x) 
= \mathsf{E} \big( f^*_K - \min (f^*_K, \xi_x)
\big| \{\xi_{x_k} = f(x_k)\}_{k=1}^K\big).
\]
In practice, the stochastic process $\xi_x$ is usually Gaussian, which allows one to numerically solve the optimization problem
\begin{equation}\label{eq:opt2}
I_{K;\{x_k, f(x_k)\}_{k=1}^K} (x)\longrightarrow \max_{x\in \mathcal D}
\end{equation} for moderate values of $K$. Namely, assume that $\xi_x$ is a centered Gaussian process with the covariance 
\[ G(x,y) = \mathsf{E}(\xi_x\xi_y).\]
Then $\xi_x$ conditioned on $\{\xi_{x_k} = f(x_k)\}_{k=1}^K$ is also a Gaussian random variable:
\[
\xi_x|\{\xi_{x_k} = f(x_k)\}_{k=1}^K\sim\mathcal N(
m_{x; \{x_k, f(x_k)\}_{k=1}^K}, \sigma^2_{x; \{x_k\}_{k=1}^K}),
\] where $m$ and $\sigma^2$ denote the conditional mean and variance. 
Note that since the process is Gaussian, the variance depends on $\{x_k\}_{k=1}^K$ but not on $\{f(x_k)\}_{k=1}^K.$   A straightforward calculation shows that 
\begin{align}
\label{eq:g}
m_{x; \{x_k, f(x_k)\}_{k=1}^K} & = \mathbf g_{K,x}^t\mathbf G_K^{-1}\mathbf f_K,\\
\label{eq:g2}
\sigma^2_{x; \{x_k\}_{k=1}^K} & = G(x,x) - \mathbf g_{K,x}^t\mathbf G_K^{-1}\mathbf g_{K,x},
\end{align} 
where
\begin{align*}
\mathbf f_K        &= (f(x_1),\ldots,f(x_K))^t,\\
\mathbf g_{K,x}  &= (G(x,x_1),\ldots,G(x,x_K))^t,\\
\mathbf G_K      &= (G(x_k,x_l))_{k,l=1}^K.
\end{align*}
Throughout the paper, we will assume that the kernel $G(x,y)$ is strictly positive definite, which in particular ensures that $\mathbf G_K$ in \eqref{eq:g},\eqref{eq:g2} is invertible. 

If $G$ is continuous, then the conditional mean and variance continuously depend on $x$, which implies  existence of the maximum in \eqref{eq:opt2}. The maximum can be attained at more than one point; any of them can be taken as $x_{K+1}$. Up to this ambiguity, the EI algorithm is completely determined by the kernel $G$.

Note that if the kernel $G$ is strictly positive definite, then 
\[
I_{K;\{x_k, f(x_k)\}_{k=1}^K} (x) 
\begin{cases}
=0, & x\in\{x_k\}_{k=1}^K,\\
>0, & x\notin\{x_k\}_{k=1}^K,
\end{cases}
\]
so that the maximizer $x_{K+1}\notin \{x_k\}_{k=1}^K$, i.e., all the points of the trajectory $\{x_k\}_{k=1}^\infty$ are different.

Consider the Hilbert space $L^2(\Omega, \mathsf P)$, where $(\Omega, \mathsf P)$ is the probability space on which the process $\xi_x$ is defined.  
Then one can geometrically interpret the conditional variance $\sigma^2_{x; \{x_k\}_{k=1}^K}$ as the squared distance between the vector $\xi_x$ and the linear span of the vectors $\{\xi_{x_k}\}_{k=1}^K$ in $L^2(\Omega, \mathsf P)$.
  
Practical implementations of the EI algorithm often use somewhat more complex modelling than described above, based on kriging \cite{Jones:1998:EGO:596070.596218}. 
This approach includes additional polynomial trends in the model; also, the covariance function is assumed to depend on a few parameters which are adjusted at each iteration using cross-validation or maximum likelihood estimates. We will not consider these complications in this paper. 

We will fix a kernel $G$ and will treat the EI algorithm described above as ideally implemented with this kernel, in the sense that the auxiliary problem \eqref{eq:opt2} is assumed to be exactly solved at each iteration $K$.
We will  then be interested in the convergence properties of the resulting sequences $x_K$ and $f(x_K)$.  
 
\paragraph{Previous rigorous results.}
EI is a popular approach to global optimization in modern engineering applications, but not much has been proved about it rigorously. 
If $\xi_x$ is the Wiener process or its stationary version, the Ornstein-Uhlenbeck process, on a segment in $\R$, 
then, using the Markov property, it is not hard to check that the EI optimization is consistent for continous objective functions, see \cite{Locatelli:1997:BAO:596056.596117}. 
In \cite{Vazquez20103088,VAZQUEZ:2010:HAL-00440827:1}, Vazquez and Bect considered the general  case of compact subsets of $\R^n$ 
and proved the convergence of the EI algorithm for sufficiently ``rough'' stationary processes $\xi_x$, 
on objective functions $f$ from the reproducing-kernel Hilbert space (RKHS) associated with $\xi_x$. 

As an intermediate step in their proof, these authors consider what they call the \emph{No-Empty-Ball} (NEB) property of the process $\xi$:

\begin{defin}
 The process $\xi_x$ is said to have the NEB property if for all sequences $\{x_k\}_{k=1}^{\infty}$ (not necessarily given by \ref{eq:history}) and all points $x$ in $\mathcal D$ the following two conditions are equivalent:
\begin{enumerate}
\item $x$ belongs to the closure of $\{x_k\}_{k=1}^{\infty}$;
\item the conditional variance $\sigma^2_{x; \{x_k\}_{k=1}^K}\to 0$  as $K\to\infty$.
\end{enumerate}
\end{defin}
The first condition clearly implies the second for processes with a continuous covariance function, but the opposite direction is more subtle.  
In particular, Vazquez and Bect  prove that the NEB property is violated by Gaussian processes with a Gaussian covariance function.   
They show, however, that the NEB property holds for a stationary process provided its spectral density goes to zero sufficiently slowly, namely if its inverse is polynomially bounded. 
Additionaly, they show that if a Gaussian process has the NEB property and the objective function is from the corresponding RKHS, then the optimization trajectory is dense in $\mathcal D$, and hence optimization is consistent on this space.

Vazquez and Bect also show that for Gaussian processes with the NEB property the optimization trajectory is dense almost surely, 
if the optimized function is a realization of the process.  

Recently, Bull \cite{bull} has obtained rigorous convergence rates for objective functions from the RKHS of the process. 

Some rigorous results are also available for certain stochastic optimization algorithms different from but closely related to EI, see, e.g., Gutmann \cite{Gutmann:2001:RBF:596093.596381}.

Finally, though in this article we don't consider covariance functions with adaptively adjusted parameters, 
we mention that these more general kinds of EI optimization are known to be inconsistent in some cases \cite{Locatelli:1997:BAO:596056.596117, bull}.

\section{Results}
As discussed above, the existing rigorous results about convergence of the EI optimization are mostly proofs of convergence under certain assumptions,
namely when the NEB property holds and/or the objective function belongs to the RKHS associated with the process.
At the same time, little is known rigorously about (in)consistency of the EI optimization when these assumptions are violated, though, for example, 
the Gaussian kernel is one of the most common kernels used in practical modelling \cite{forrester},
while in engineering applications one rarely expects strong regularity of the objective function.

Vazquez and Bect \cite{Vazquez20103088,VAZQUEZ:2010:HAL-00440827:1} conjecture consistency for all continuous objective functions provided the process has the NEB property. The result of Locatelly \cite{Locatelli:1997:BAO:596056.596117} confirms this in the case of the  Wiener process.

The goal of this paper is to examine convergence of the EI algorithm for analytic Gaussian processes. 
More precisely, we will consider kernels with spectral densities which very rapidly converge to 0; this property is related to analyticity by Paley-Wiener--type theorems (see, e.g., \cite{katznelson2004introduction}, page 209). 
Our main result is a class of examples demonstrating some lack of consistency of the EI optimization in this case, for objective functions which are  not analytic.
We thus show, in particular, that the EI optimization cannot be fully consistent if both the NEB and RKHS assumptions are dropped. 

We will consider only 1D models in this paper and let $\mathcal D = [-1,1]$. 
We consider a translation invariant covariance $G$, i.e. 
\[G(x',x'')=G(x'-x'',0)\equiv G(x'-x''),\] 
and assume that it has a spectral density $\whg(t), t\in\R$:
\[G(x) = \int_{\R}\whg(t)e^{itx}dt,\quad 
\whg(t) = \frac{1}{2\pi}\int_{\R}G(x)e^{-itx}dx.\]
Since $G$ is real and even, such is $\whg$: $\whg(t)=\whg(-t)\in\R$. 
We assume that  $\whg(t)>0$ for all $t$, so that the kernel $G$ is strictly positive definite.

We start by showing, as a preparation for the main result, that if \begin{equation}\label{eq:exp2}\whg(t)\le c_0e^{-c|t|}, \end{equation}
with some $c_0,c>0$, then the process $\xi_x$ does not have the NEB property. 
Condition \eqref{eq:exp2} implies, in particular, that $G$ is analytic in the strip $|\operatorname{Im} (z)|<c$.
  
\begin{theor}\label{thm:1}
Let $(\xi_x)_{x\in [-1,1]}$ be a centered stationary Gaussian process defined on a probability space $(\Omega, \mathsf P)$. Suppose that the spectral density $\whg$ of the process satisfies condition \eqref{eq:exp2} with some $c_0,c>0$. Let $A$ be any infinite subset of the segment $[-1,1]$.  Then all the random variables $(\xi_x)_{x\in [-1,1]}$ belong to the closed linear span of the random variables $(\xi_y)_{y\in A} $ in $L^2(\Omega,\mathsf P)$.     
\end{theor}

We next indicate a class of optimization problems where the EI optimization trajectory is provably not dense in $[-1,1]$.  

In proving this result, we have found especially useful a pair of asymptotic bounds for the conditional variance, which we will state now as a separate theorem.
 
Suppose that the spectral density is represented in the form
\begin{equation}
\label{eq:s0}
\whg(t)=e^{-S(|t|)}=e^{-T(\ln|t|)}
\end{equation}
with some functions $S\in C^2(\R_+), T\in C^2(\R)$.
We will assume that 
\begin{equation}
S'(t),S''(t)\ge 0\quad \text{ for } t>0,\label{eq:s1}
\end{equation}
and
\begin{equation}
S'(t)\to +\infty\quad\text{ as } t\to +\infty.\label{eq:s2}
\end{equation}
Condition \eqref{eq:s1} implies, in particular, that $T$ is convex, since
\begin{equation}\label{eq:t00}
T''(s)=(S(e^s))''=e^sS'(e^s)+e^{2s}S''(e^s)\ge 0.
\end{equation}
Also, \begin{equation}\label{eq:t1}
T'(s)\to +\infty\quad\text{ as } s\to +\infty,
\end{equation} 
since
\begin{equation}\label{eq:ts}
T'(s) = e^sS'(e^s).
\end{equation}

Let $T^*$ be the Legendre transform of $T$:
\[T^*(q) = \max_{s\in\R}(qs-T(s)).\]
Then, by \eqref{eq:t1}, $T^*(q)$ is finite for all sufficiently large $q$; the point $s^*$ where the maximum is attained satisfies the condition $T'(s^*)=q$. 

\begin{theor}\label{thm:2}
Suppose that the spectral density of the covariance function $G$ is represented in the form \eqref{eq:s0} so that  conditions \eqref{eq:s1},\eqref{eq:s2} hold. Then, for sufficiently large $K$, the following inequalities hold for any $K+1$ different points $x,x_1,\ldots,x_K\in[-1,1]$:
\begin{equation}\label{eq:th2main}
e^{-K}\le
 \frac{\sigma^2_{x; \{x_k\}_{k=1}^K}}{e^{F(K)}\prod_{k=1}^K|x-x_k|^2}\le e^{2K},
\end{equation} 
where \[F(K)=T^*(2K+1)-(2K+1)\ln K.\] 
Furthermore, $F(K)$ monotonically decreases for sufficiently large $K$, and
\begin{equation}\label{eq:f1}
\frac{F(K)}{K}\to -\infty \text { as } K\to+\infty.
\end{equation}
\end{theor} 

An example of a  family of spectral densities covered by this theorem is
\begin{equation}\label{eq:cov}
\whg(t) =e^{-a|t|^b}
\end{equation}
with $a>0,b>1$. In particular,  with $b=2$ this gives the Gaussian covariance functions
\begin{equation}
\label{eq:gauss}
G(x)=\frac{1}{2\sqrt{\pi a}}e^{-\frac{x^2}{4a}}.
\end{equation} 
Spectral densities \eqref{eq:cov} correspond to 
\begin{align*}
S(|t|) &= a|t|^b,\\
T(s) &= ae^{bs},
\end{align*}
so that conditions \eqref{eq:s1},\eqref{eq:s2} hold for all $a>0,b>1$. We find in this case
\begin{align*}
T^*(q) &=\frac{q}{b}\left(\ln\frac{q}{ab}-1\right),\\
F(K) &= \frac{2K+1}{b}\left(\ln\frac{2K+1}{ab}-b\ln K-1\right).
\end{align*}

We state now our main result on the EI optimization. 
Recall that if $G$ is a positive definite kernel, then $G(0)=\max_{x\in\R} G(x)$.

\begin{theor}\label{thm:3}
Under assumptions of Theorem \ref{thm:2}, consider optimization problem \eqref{eq:mainopt} on $\mathcal D = [-1,1]$ with the objective function \[f = -G.\]
Suppose that the EI optimization with the kernel $G$ starts from the point $x_1=0$ (i.e., the point where the minimum is already attained). Then the optimization trajectory $\{x_k\}_{k=1}^{\infty}$ converges to 0; in particular the trajectory is not dense in $[-1,1]$. Moreover,  for sufficiently large $K$
\begin{equation}
\label{eq:thm3}
e^{2^K F(K)} \le |x_{K+1}| \le e^{F(K)/3},
\end{equation} 
where $F$ is as defined in Theorem \ref{thm:2}.
\end{theor} 

As pointed out in Proposition \ref{thm:prop1}, Theorem \ref{thm:3} implies that there are continuous objective functions for which the EI optimization is inconsistent. Such functions can be obtained by modifying the objective function $-G$ on a set not containing points of the corresponding trajectory $\{x_k\}_{k=1}^{\infty}$. 
Recall that under assumptions of Theorems \ref{thm:2} and \ref{thm:3} the kernel $G$ is analytic. We cannot modify $-G$ preserving its analyticity, but can modify it preserving its infinite smoothness. Recalling the family of examples discussed above, we obtain, in particular, the following corollary regarding the  Gaussian covariance function.

\begin{corol}
For any Gaussian covariance function \eqref{eq:gauss}, there exists an objective function $f\in C^{\infty}([-1,1])$ such that the EI optimization of $f$ starting with $x_1=0$ is not consistent. 
\end{corol}

We briefly discuss now practical implications of these results.

On the one hand, there are certain caveats to their practical interpretation. First, we consider only the simplest version of the EI optimization in 1D, while real applications are mostly higher-dimensional. Second, realistic optimization budgets may be too low in many problems for the indicated asymptotic behavior to be relevant. Third, the theoretical consistency, as such, may in principle be restored by trivial adjustments of the algorithm, e.g., by occasionally alternating the EI trajectory with a fixed dense sequence in $\mathcal D$.  

Nevertheless, our results suggest that, in general, EI algorithms with analytic kernels are not reliable beyond a narrow class of very smooth functions -- at least hard to justify theoretically. Moreover, they may be prone to early ill-conditioning and numerical instabilities due to excessive accumulation of trajectory points (see also the numerical example below). It appears that for practical numerical optimization of generic objective functions, if EI is to be applied, then a more reliable choice for the covariance would be a rough kernel with an inverse polynomial falloff of the spectral density, for example from the Mat\'ern family (see, e.g., \cite{rasmussen}). 

In the next section we report a numerical test of Theorem \ref{thm:3}.
Then, in sections \ref{sec:thm1}--\ref{sec:thm3} we provide the proofs of Theorems  \ref{thm:1}--\ref{thm:3}. 

%-------------------------------------------------------------------------------------------------

\section{A numerical example} 
To confirm Theorem \ref{thm:3}, we report direct numerical results of the EI optimization of the objective function $f(x)=-e^{-x^2}$ performed with the kernel $G(x) = e^{-x^2}$. 

In short,  our numerical procedure is as follows. 
At each iteration,  for any trial point $x$ we compute the  parameters $m_x, \sigma_x^2$ of the associated posterior Gaussian variable by explicitly using formulas \eqref{eq:g},\eqref{eq:g2}. 
We then compute the expected improvement at $x$ using the well-known formula (see \cite{JoScWe97a})
\begin{equation*}
I_K(x)=(f^*_K-m_x)\Psi\Big(\frac{f^*_K-m_x}{\sigma_x}\Big)+
\sigma_x\psi\Big(\frac{f^*_K-m_x}{\sigma_x}\Big),
\end{equation*}
where $\psi$ and $\Psi$ and the standard normal density and cumulative distribution function, respectively.
To optimize $I_K(x)$ over $x$, we simply sample $x$ uniformly on a logarithmic scale: precisely, we try $x=\pm e^{-l\epsilon}$, where $\epsilon=0.02$ and $l=0,1,\ldots, 10^4$. 

We should point out that this numerical procedure is quite unstable for our kernel and objective function. 
As the posterior variance of the process rapidly converges to 0 and the trajectory $\{x_K\}$ to $x_1=0$,  computation of the expected improvement involves, in several places, subtraction of almost equal quantities, in particular in \eqref{eq:g2}. 
Also, the matrices $\mathbf G_K$ quickly get ill-conditioned. 
As a result, precision of, for example, the usual ``double''  floating point format, which has the 53-bit significand (approximately 16 decimal digits), is exhausted very soon during this optimization. 
For this reason, we perform our test with the extended precision of 300 decimal digits, using the free library \textsc{Mpmath} \cite{mpmath} for that purpose.
 
The first 10 elements of the trajectory appear then to be reasonably reliably computed, and are shown in Table \ref{fig:numex},  together with the corresponding expected improvements. 
This result confirms Theorem \ref{thm:3}, also suggesting the actual asymptotic of $x_K$ is closer to the lower rather than upper bound in \eqref{eq:thm3}.
%confirming Theorem \ref{thm:3}.

\begin{table}[htb]\small
\begin{center}
\begin{tabular}{ccc}\toprule
$K$ & $x_{K}$ & $I_{K-1;\{x_k, f(x_k)\}_{k=1}^{K-1}} (x_K)$\\\midrule
  1 &              0 &     --- \\
  2 &        -0.63 &    0.16\\
  3 &         0.77 &    0.13\\
  4 &         0.23 &   0.025\\
  5 &          -0.1 &  0.0013\\
  6 &     0.0036 & 3.4e-06\\
  7 &   -7.3e-06 & 1.4e-11\\
  8 &    2.8e-11 & 2.2e-22\\
  9 &   -4.1e-22 & 4.5e-44\\
 10 &    7.9e-44 & 1.7e-87\\\bottomrule
\end{tabular}\caption{The first 10 elements of the EI optimization trajectory with the respective expected improvements, for the kernel $G(x)=e^{-x^2}$ and the objective function $f=-G$.}  \label{fig:numex}
\end{center}
\end{table}

%-------------------------------------------------------------------------------------------------

\section{Proof of Theorem \ref{thm:1}}  \label{sec:thm1}

Let the Hilbert space $\mathcal H$ be the closed span of the Gaussian random variables $(\xi_x)_{x\in [-1,1]}$ in $L^2(\Omega,\mathbf P)$. We use the canonical isometry between $\mathcal H$ and $L^2(\R, \whg)$:
\begin{equation}\label{eq:iso}
\xi_x\in\mathcal H \longmapsto \phi_x\in L^2(\R, \whg),
\end{equation}
where $$\phi_x(t):=e^{itx},$$ so that
$$\langle \xi_x, \xi_{x'}\rangle_{\mathcal H} = G(x-x')=\int_{\R}e^{itx}e^{-itx'}\whg(t)dt = \langle \phi_x,\phi_{x'}\rangle_{L^2(\R, \whg)}.$$
In terms of this isometry, the claim of Theorem \ref{thm:1} is that for any $x\in[-1,1]$ the function $\phi_x$ can be approximated in $L^2(\R, \whg)$  by finite linear combinations of functions $(\phi_{y})_{y\in A}$.

We first prove the following
\begin{lemma}
Let $x$ be any point in $[-1,1]$, and $\{x_k\}_{k=1}^\infty$ any infinite sequence of points in $[-1,1]$ 
such that $x_{k}\ne x_{l}$ for $k\ne l$ and $|x-x_k|<\frac{c}{4}$ for all $k$, where $c$ is from \eqref{eq:exp2}. 
Then, assuming \eqref{eq:exp2}, $\phi_x$ can be approximated in $L^2(\R, \whg)$ with arbitrary accuracy by finite linear combinations of $\phi_{x_k}$.
\end{lemma}
\begin{proof}
By the theory of polynomial interpolation (see, e.g., \cite{kincaid2002numerical}), for any positive integer $K$ we can choose coefficients $\lambda_{K,1},\ldots, \lambda_{K,K}$ such that for any polynomial $p$ with $\deg p<K$ we have 
\begin{equation}\label{eq:pi}
p(x)=\sum_{k=1}^{K}\lambda_{K,k} p(x_k),\end{equation}
namely, 
\[\lambda_{K,k} =\prod_{\substack{l=1\\ l\ne k}}^{K}\frac{x-x_l}{x_k-x_l}.\]
The r.h.s. of \eqref{eq:pi} is a polynomial in $x$ of degree $<K$. If a function $p(x)$ is not a polynomial of  degree $<K$, then the difference between the left and right sides of \eqref{eq:pi} can be interpreted as the error of polynomial interpolation and  written in terms of divided differences of $p$:
\[p(x)-\sum_{k=1}^{K}\lambda_{K,k} p(x_k) = p[x, x_1, \ldots, x_{K}]\prod_{k=1}^{K}(x-x_k).\]
By the Hermite--Genocchi formula, \[|p[x, x_1, \ldots, x_{K}]|\le\frac{\big\|\frac{d^{K}p}{dx^{K}}\big\|_{\mathrm{conv}(x,x_1,\ldots,x_{K})}}{K!},\]
where $\|\cdot\|_{\mathrm{conv}(x,x_1,\ldots,x_{K})}$ denotes the maximum over the convex hull of the points $x,x_1,\ldots,x_{K}$.
In particular, if $p(x)=e^{ixt}$ with some $t\in\R$, then 
\[|p[x, x_1, \ldots, x_{K}]|\le \frac{t^K}{K!}.\]
Accordingly,
\[\left|e^{ixt}-\sum_{k=1}^{K}\lambda_{K,k} e^{ix_kt}\right|\le \frac{t^K}{K!}\prod_{k=1}^{K}|x-x_k|\]
and hence 
\begin{equation}
\label{eq:th1_1}
\left\|\phi_x-\sum_{k=1}^K\lambda_{K,k}\phi_{x_k}\right\|^2_{L^2(\R, \whg)}
\le \int_\R\frac{t^{2K}}{(K!)^2}\whg(t)dt \prod_{k=1}^{K}|x-x_k|^2.
\end{equation}
Now recall that $\whg(t)\le c_0e^{-c|t|}$ with some $c_0,c>0$, and $|x-x_k|<\frac{c}{4}$. 
Then
\begin{align*}
\left\|\phi_x-\sum_{k=1}^K\lambda_{K,k}\phi_{x_k}\right\|^2_{L^2(\R, \whg)}
&\le \frac{2c_0}{c^{2K+1}(K!)^2}\left(\frac{c}{4}\right)^{2K}\int_0^{+\infty} s^{2K}e^{-s}ds\\
&\le\frac{2c_0 (2K)!}{c4^{2K}(K!)^2}\\
&= O(2^{-2K})\stackrel{K\to\infty}{\longrightarrow}0,
\end{align*}
where we used Stirling's formula in the last step. 
\end{proof}

Now, let $B$ denote the set of all those points $x$ in $[-1,1]$ for which $\phi_x$ can be approximated by linear combinations of $(\phi_{y})_{y\in A}$. 
We prove that $B=[-1,1]$ in several steps.

\noindent
{\em Statement 1}.  $B$ has a non-empty interior.

Indeed, since $A\subset [-1,1]$ is infinite, we can find an interval of length $\frac{c}{4}$ in $[-1,1]$ 
that contains infinitely many points of $A$. Then, by the above lemma, any point of this interval belongs to $B$.

\noindent
{\em Statement 2}.  If $x$ belongs to the interior of $B$ and $|x'-x|<\frac{c}{4}$ for some $x'\in[-1,1]$, 
then $x'$ also belongs to the interior of $B$. 

Indeed, it follows from the hypothesis that we can find infinitely many distinct points $x_k$ in $B$ 
such that $|x'-x_k|<\frac{c}{4}$ for all of them. 
By the above lemma, $\phi_{x'}$ can then be approximated  by finite linear combinations of $\phi_{x_k}$.
But, since $x_k\in B$, any $\phi_{x_k}$ can in turn be approximated  by finite linear combinations of $(\phi_{y})_{y\in A}$. 
It follows that $\phi_{x'}$ can  be approximated  by finite linear combinations of $(\phi_{y})_{y\in A}$, i.e., 
$x'\in B$.

Now, in the above argument, $x'$ could be replaced by any $x''$ sufficiently close to $x'$ so that $|x''-x|<\frac{c}{4}$ still holds,
and we would get $x''\in B$. It follows that $x'$ belongs not only to $B$, but even to the interior of $B$.

\noindent
{\em Statement 3}. $B=[-1,1]$.

This follows immediately from statements 1 and 2.

This completes the proof of Theorem 1.

%-------------------------------------------------------------------------------------------------  
  
\section{Proof of Theorem \ref{thm:2}}

We use again the canonical isometry \eqref{eq:iso} to express the conditional variance $\sigma^2_{x; \{x_k\}_{k=1}^K}$
 as
\[
\sigma^2_{x; \{x_k\}_{k=1}^K} =  \min_{\lambda_1,\ldots,\lambda_K}\int_{\R}\bigg|e^{ixt}-\sum_{k=1}^K\lambda_ke^{ix_kt}\bigg|^2\whg(t)dt.
\]

We start now with the proof of the upper bound in \eqref{eq:th2main}.
We already know from the proof of Theorem \ref{thm:1} that (see \eqref{eq:th1_1}) 
\begin{equation}\label{eq:sig1}
\sigma^2_{x; \{x_k\}_{k=1}^K}\le
\frac{1}{(K!)^2}\prod_{k=1}^{K}|x-x_k|^2\int_\R t^{2K}\whg(t)dt.
\end{equation}
We substitute $t=\pm e^s$ in the integral on the r.h.s.:
\begin{equation}\label{eq:sig2}
\int_\R t^{2K}\whg(t)dt = 2\int_\R e^{(2K+1)s-T(s)} ds.
\end{equation}
We can now derive an upper bound for this integral using a basic form of the Laplace method.
Consider the function $\widetilde T_K(s) := (2K+1)s-T(s)$ which is concave by
\eqref{eq:t00}.
Let
\begin{equation}\label{eq:sk}
s_K^*=\arg\max \widetilde T_K(s).
\end{equation}
Using $\widetilde T'_K(s_K^*)=0$ and $\widetilde T''_K = -T''$, we can write 
\begin{align}
\label{eq:tkchi}
\widetilde T_K(s)
&=\widetilde T_K(s_K^*)+\widetilde T'_K(s_K^*)(s-s_K^*)+\int_{s_K^*}^s\left(\int_{s_K^*}^{s_1} \widetilde T''_K(s_2)ds_2\right)ds_1\nonumber\\
&=\widetilde T_K(s_K^*)-\int_{s_K^*}^s\left(\int_{s_K^*}^{s_1} T''(s_2)ds_2\right)ds_1\nonumber\\
&\le \widetilde T_K(s_K^*)-\int_{s_K^*}^s\left(\int_{s_K^*}^{s_1} \chi''(s_2-s_K^*)ds_2\right)ds_1                 \nonumber\\
&= \widetilde T_K(s_K^*)-\chi(s-s_K^*)  \nonumber\\
&= T^*(2K+1)-\chi(s-s_K^*),\quad\text{ for all } s\in\R,
\end{align}
for any $C^2$ function $\chi$ such that $\chi(0)=\chi'(0)=0$ and 
\begin{equation}\label{eq:chi''}
\chi''(s)\le T''(s_K^*+s)\quad\text{for all }s.
\end{equation}
It follows then from \eqref{eq:tkchi} that
\[
\int_\R e^{(2K+1)s-T(s)} ds\le c_0 e^{T^*(2K+1)},
\]
where $c_0=\int_{\R}e^{-\chi(s)}ds$.
By \eqref{eq:s2},\eqref{eq:t00}, 
$T''(s)\stackrel{s\to+\infty}{\longrightarrow}+\infty$,
so we can choose a $\chi$ such that $c_0=\frac{1}{2}$ while \eqref{eq:chi''} and hence \eqref{eq:tkchi} hold for all sufficiently large $K$; for example 
\begin{equation*}
\chi(s) = c_1\cdot
\begin{cases}
6s^2-s^4, & |s|\le 1,\\
8|s|-3, &  |s|>1,
\end{cases}
\end{equation*}
with the appropriate constant $c_1$.
 Using Stirling's formula, we then get from \eqref{eq:sig1},\eqref{eq:sig2}, for sufficiently large $K$,
\[
\sigma^2_{x; \{x_k\}_{k=1}^K}\le
e^{T^*(2K+1)-(2K+1)\ln K+2K}
\prod_{k=1}^{K}|x-x_k|^2,
\]
which is the upper bound in \eqref{eq:th2main}.

To prove the lower bound in \eqref{eq:th2main}, we will use the following lemma.
\begin{lemma}\label{th:lower}
Let $z\in \C$ and $\{z_k\}_{k=1}^K\subset\C$. Let 
\[v=(1,z,z^2,\ldots,z^{K})\in\C^{K+1}.\] 
Similarly, let
\[v_k=(1,z_k,z_k^2,\ldots,z_k^{K})\in\C^{K+1},\quad k=1,\ldots,K.\]   
Then the standard $l^2$ distance $\rho$ in $\C^{K+1}$ between $v$ and the linear span of $\{v_k\}_{k=1}^K$  equals 
\begin{equation}\label{eq:rho}\rho = \frac{\prod_{k=1}^K|z-z_k|}{\Big(1+\sum_{q=1}^K\big|\sum_{1\le k_1<\ldots<k_q\le K}\prod_{t=1}^q z_{k_t}\big|^2\Big)^{1/2}}.\end{equation}
\end{lemma}
\begin{proof}
We have 
\begin{equation}\label{eq:rhog}
\rho^2 = \frac{g(v,v_1,\ldots,v_K)}{g(v_1,\ldots,v_K)},
\end{equation} where $g(\cdot)$ denotes the Gram determinant of the given system of vectors. 
Since $v_k$ and $v$ are $(K+1)$-dimensional,  $g(v,v_1,\ldots,v_K)$ can be computed simply from the Vandermonde determinant for $z, z_1,\ldots, z_K$: 
\begin{equation}\label{eq:vand1}
g(v,v_1,\ldots,v_K) = \prod_{0\le k<l \le K}|z_k-z_l|^2,
\end{equation}
where we have denoted $z_0\equiv z$. 
In order to compute $g(v_1,\ldots,v_K),$ we note first that it can be expressed, 
by the Cauchy-Binet formula, as 
\begin{equation}\label{eq:vand2}
g(v_1,\ldots,v_K) = \sum_{s=0}^K |\Delta_{K,s}|^2,
\end{equation} 
where $\Delta_{K,s}$ is the $K\times K$ minor of the $K\times(K+1)$ matrix $(z_k^t)_{k=1, t=0}^{K,K}$ obtained by removing the row $(z_k^s)_{k=1}^K$. Note that $\Delta_{K,K}$ is the usual Vandermonde determinant, and $\Delta_{K,0}=\Delta_{K,K}\prod_{k=1}^K z_k$. 
We can compute  $\Delta_{K,s}$ for any $s$ in a way similar to the usual inductive evaluation of the Vandermonde determinant. Namely, define for brevity 
\begin{equation*}
\mu_s(t) =
\begin{cases}
t,      & t<s;\\
t+1, & t\ge s.
\end{cases} 
\end{equation*}
Then for  $0<s<K$ we have, performing linear transformations with rows and columns,
\begin{align}\label{eq:detdelta1}
\Delta_{K,s} &= \det\Big(z_k^{\mu_s(t)}\Big)_{\substack{1\le k\le K\\ 0\le t\le K-1}}\nonumber\\
                     &= \det\left(
                               \begin{cases}
                               z_k^{\mu_s(t)}-z_K^{\mu_s(t)}, & k<K\\
                               z_K^{\mu_s(t)},                         & k=K
                               \end{cases}
                               \right)
                               _{\substack{1\le k\le K\\ 0\le t\le K-1}}\nonumber\\   
                     &= (-1)^{K-1}\det
                               \Big(
                               z_k^{\mu_s(t)}-z_K^{\mu_s(t)}
                               \Big)
                               _{\substack{1\le k\le K-1\\ 1\le t\le K-1}}\nonumber\\                                       
                     &= \det
                               \Bigg(\sum_{\tau_t=0}^{\mu_s(t)-1}
                               z_k^{\tau_t}z_K^{\mu_s(t)-1-\tau_t}
                               \Bigg)
                               _{\substack{1\le k\le K-1\\ 1\le t\le K-1}}\prod_{k=1}^{K-1}(z_{K}-z_k)\nonumber\\      
                     &= \det
                               \Bigg(\sum_{\tau_t=\mu_s(t-1)}^{\mu_s(t)-1}
                               z_k^{\tau_t}z_K^{\mu_s(t)-1-\tau_t}
                               \Bigg)
                               _{\substack{1\le k\le K-1\\ 1\le t\le K-1}}\prod_{k=1}^{K-1}(z_{K}-z_k)\nonumber\\    
                      &= \det
                               \left(
                               \begin{cases}
                               z_k^{\mu_s(t)-1},            & t\ne s\\
                               z_k^{s-1} z_K+z_k^{s},  & t=s
                               \end{cases}
                               \right)
                               _{\substack{1\le k\le K-1\\ 1\le t\le K-1}}\prod_{k=1}^{K-1}(z_{K}-z_k)\nonumber\\    
                       &= (z_K\Delta_{K-1,s}+\Delta_{K-1,s-1})\prod_{k=1}^{K-1}(z_K-z_k).
\end{align}
Similar identities hold if $s=0$ or $s=K$, but with one of the terms $\Delta_{K-1,s-1}$, $z_K\Delta_{K-1,s}$ missing:
\begin{equation}\label{eq:detdelta2}
\Delta_{K,0}=z_K\Delta_{K-1,0}\prod_{k=1}^{K-1}(z_K-z_k), 
\quad \Delta_{K,K}=\Delta_{K-1,K-1}\prod_{k=1}^{K-1}(z_K-z_k). 
\end{equation}
Iterating identities \eqref{eq:detdelta1},\eqref{eq:detdelta2} $K$ times, we get
\[\Delta_{K,s} = \prod_{1\le k<l \le K}(z_l-z_k)\sum_{1\le k_1<\ldots<k_{K-s}\le K}\prod_{t=1}^{K-s} z_{k_t}.\]
Substituting this equality in \eqref{eq:vand2} and combining with  \eqref{eq:rhog} and \eqref{eq:vand1}, we get \eqref{eq:rho}  with $q=K-s$.
\end{proof} 

To derive now the lower bound in \eqref{eq:th2main}, fix a $t_0=t_0(K)>0$, to be chosen later. Using monotonicity of $\whg(t)$ for $t\ge 0$, which follows from \eqref{eq:s1}, we write:
\begin{align*}
 \sigma^2_{x; \{x_k\}_{k=1}^K} &=
\min_{\lambda_1,\ldots,\lambda_K}\int_{\R}\bigg|e^{ixt}-\sum_{k=1}^K\lambda_ke^{ix_kt}\bigg|^2\whg(t)dt\\
&\ge \whg\left(\frac{(K+1)t_0}{2}\right)\min_{\lambda_1,\ldots,\lambda_K}\int_{-\frac{(K+1)t_0}{2}}^{\frac{(K+1)t_0}{2}}\bigg|e^{ixt}-\sum_{k=1}^K\lambda_ke^{ix_kt}\bigg|^2 dt\\
& = \whg\left(\frac{(K+1)t_0}{2}\right) \min_{\lambda_1,\ldots,\lambda_K}
\int_{-\frac{(K+1)t_0}{2}}^{-\frac{(K-1)t_0}{2}}\sum_{l=0}^K
\bigg|e^{ixt}e^{ilxt_0}-\sum_{k=1}^K\lambda_k e^{ix_kt}e^{ilx_kt_0}\bigg|^2 dt\\
& \ge \whg\left(\frac{(K+1)t_0}{2}\right) \int_{-\frac{(K+1)t_0}{2}}^{-\frac{(K-1)t_0}{2}} \min_{\lambda_1,\ldots,\lambda_K}%\Bigg(
\sum_{l=0}^K
\bigg|e^{ilxt_0}-\sum_{k=1}^K\lambda_k e^{i(x_k-x)t}e^{ilx_kt_0}\bigg|^2 %\Bigg)
dt\\
& = \whg\left(\frac{(K+1)t_0}{2}\right) t_0
\min_{\lambda_1,\ldots,\lambda_K}
\sum_{l=0}^K
\bigg|e^{ilxt_0}-\sum_{k=1}^K\lambda_k e^{ilx_kt_0}\bigg|^2\\
& = \whg\left(\frac{(K+1)t_0}{2}\right) t_0\rho^2,
\end{align*}
where $\rho$ is the distance defined as in Lemma \ref{th:lower} for $z=e^{ixt_0}, z_k=e^{ix_kt_0}.$ Since $|z|=|z_k|=1$, the denominator in \eqref{eq:rho} is bounded from above by $2^K$. Therefore,
\[
\sigma^2_{x; \{x_k\}_{k=1}^K}\ge\whg\left(\frac{(K+1)t_0}{2}\right)\frac{t_0}{2^{2K}}\prod_{k=1}^K|e^{ixt_0}-e^{ix_kt_0}|^2.
\]
Let us assume that 
\begin{equation}\label{eq:t0}
t_0<\frac{\pi}{2}.
\end{equation}
In this case, since $x,x_k\in[-1,1]$, we have $\frac{|x-x_k|t_0}{2}<\frac{\pi}{2}$, hence
\[ |e^{ixt_0}-e^{ix_kt_0}|=2\sin\frac{|x-x_k|t_0}{2}\ge \frac{4}{\pi}\frac{|x-x_k|t_0}{2}=\frac{2|t_0|}{\pi}|x-x_k|.\]
Therefore, assuming \eqref{eq:t0}, 
\begin{equation}\label{eq:sigma1}
\sigma^2_{x; \{x_k\}_{k=1}^K}\ge\whg\left(\frac{(K+1)t_0}{2}\right)\frac{t_0^{2K+1}}{\pi^{2K}}\prod_{k=1}^K|x-x_k|^2.
\end{equation}
Now let us choose $t_0$ so that 
\[\frac{(K+1)t_0}{2}=e^{ s_K^*},\]
where $s_K^*$ is given by \eqref{eq:sk}.
Then, if \eqref{eq:t0} holds, we get from \eqref{eq:sigma1}
\begin{align*}
\sigma^2_{x; \{x_k\}_{k=1}^K}
&\ge e^{-T(s_K^*)} e^{(2K+1)s_K^*} \frac{2^{2K+1}}{(K+1)^{2K+1}\pi^{2K}} \prod_{k=1}^K|x-x_k|^2\\
&= e^{T^*(2K+1)-(2K+1)\ln(K+1)}\frac{2^{2K+1}}{\pi^{2K}}\prod_{k=1}^K|x-x_k|^2.
\end{align*}
This implies the lower bound in \eqref{eq:th2main}, since $\frac{4}{\pi^2}>\frac{1}{e}$.
We have to check, however, that condition \eqref{eq:t0} is fulfilled. 
The value $s_K^*$ satisfies the condition $2K+1 = T'(s_K^*)=e^{s_K^*}S'(e^{s_K^*})$.
Since $S'(t)\to +\infty$ as $t\to +\infty$, it follows that 
\begin{equation}\label{eq:sko}
e^{s_K^*} = o(2K+1)\quad \text{ as }K\to\infty. 
\end{equation}
Therefore $t_0\to 0$ as $K\to\infty$, so \eqref{eq:t0} is fulfilled for sufficiently large $K$.  

We now prove \eqref{eq:f1}. Since $T(s_K^*)\ge 0$ for sufficiently large $K$, we have
\[\frac{F(K)}{2K+1}=
\frac{T^*(2K+1)-(2K+1)\ln K}{2K+1} 
= s_K^*-\frac{T(s_K^*)}{2K+1}-\ln K
\le s_K^*-\ln K\stackrel{K\to\infty}{\longrightarrow} -\infty, 
\]
where we used \eqref{eq:sko} in the last step. 

It remains to prove that $F(K)$ monotonically decreases for sufficiently large $K$. We want to show that 
\[\frac{dF(K)}{dK} = 2(T^*)'(2K+1)-2\ln K-\frac{2K+1}{K}<0.\]
It suffices to show that 
\[(T^*)'(2K+1)-\ln (2K+1)\to -\infty,\text{ as } K\to +\infty.\]
By duality of the Legendre transform, this is equivalent to    
\[s-\ln(T'(s))\to -\infty,\text{ as } s\to +\infty,\]
which follows from \eqref{eq:ts},\eqref{eq:s2}.
%-------------------------------------------------------------------------------------------------  

\section{Proof of Theorem \ref{thm:3}}  \label{sec:thm3}
In this section, $\{x_k\}_{k=1}^{\infty}$ denotes the optimization trajectory obtained by \eqref{eq:history},\eqref{eq:ai} with $x_1=0$.

We start proving Theorem \ref{thm:3} by first noting that, under the hypotheses of the theorem, the mean expected value of the objective function $f=-G$ at each point of $[-1,1]$ is exactly equal to its actual value, throughout the whole optimization process:
\[
m_{x; \{x_k, f(x_k)\}_{k=1}^K} = f(x),\quad \forall x\in[-1,1],\forall K\ge 1.
\] 
Indeed, by \eqref{eq:g}, $m_{x; \{x_k, f(x_k)\}_{k=1}^K}$ is the unique interpolant of the function $f$ at the points $x_1,\ldots,x_K$ having the form $\sum_{k=1}^K\lambda_k G(x-x_k)$ with some coefficients $\lambda_k$.  But $f(x)=-G(x-x_1)$ is of this form, so it is equal to the interpolant.

Since $f$ attains its minimum at $x_1=0$, we have $f^*_K=f^*=-G(0)$ for all $K$, hence  the expected improvement can be written as 
\begin{align} 
I_{K;  \{x_k, f(x_k)\}_{k=1}^K}  (x) 
&= \mathsf{E} \big( -G(0) - \min (-G(0), \xi_x)
\big| \{\xi_{x_k} = f(x_k)\}_{k=1}^K\big) \nonumber \\
&=\frac{1}{\sqrt{2\pi}\sigma_{x; \{x_k\}_{k=1}^K}}\int_{-\infty}^{-G(0)}
\exp\left\{-\frac{(t+G(x))^2}{2\sigma^2_{x; \{x_k\}_{k=1}^K}}\right\}
(-G(0)-t)dt \nonumber \\
&=\frac{\sigma_{x; \{x_k\}_{k=1}^K}}{\sqrt{2\pi}}\int_0^{\infty}
\exp\left\{-\frac{1}{2}\bigg(w+\frac{G(0)-G(x))}{\sigma_{x; \{x_k\}_{k=1}^K}}\bigg)^2\right\}
wdw \label{eq:ie3}
\end{align}
\begin{lemma}\label{lm:3} For any $h\ge 0$
\[ \frac{1}{2}e^{-h^2} \le \int_0^{\infty} e^{-\frac{(w+h)^2}{2}}wdw \le e^{-\frac{h^2}{2}}. \]
\end{lemma}
\begin{proof} On the one hand,
\[ \int_0^{\infty} e^{-\frac{(w+h)^2}{2}}wdw 
\le  \int_0^{\infty} e^{-\frac{w^2}{2}}e^{-\frac{h^2}{2}}wdw=e^{-\frac{h^2}{2}},\]
where we have used $hw\ge 0$. On the other hand, 
\[ \int_0^{\infty} e^{-\frac{(w+h)^2}{2}}wdw 
\ge  \int_0^{\infty} e^{-w^2}e^{-h^2}wdw=\frac{1}{2}e^{-h^2},\]
where we have used $hw\le\frac{w^2+h^2}{2}$.
\end{proof}

In the sequel, we shorten the notation for  the expected improvement to $I_K(x)$. 

\begin{lemma} Under the assumptions of Theorem \ref{thm:3}, there exist constants $c_1,c_2>0$, depending only on the kernel $G$, such that  the  following assertions hold for $K$ large enough.
\begin{enumerate}
\item For all $x\in[-1,1]$  
\begin{align}
I_K(x) \label{eq:expz}
&\ge \exp\bigg\{-c_1e^{K-F(K)} \frac{x^2}{\prod_{k=2}^K|x_k-x|^2}\bigg\}
\frac{e^{(F(K)-K)/2}}{2\sqrt{2\pi}} |x|\prod_{k=2}^K|x_k-x|\\
I_K(x) \label{eq:expz1}
&\le \exp\bigg\{-c_2e^{-2K-F(K)} \frac{x^2}{\prod_{k=2}^K|x_k-x|^2}\bigg\}
\frac{e^{F(K)/2+K}}{\sqrt{2\pi}} |x|\prod_{k=2}^K|x_k-x|
\end{align}
\item
If, additionally, 
\begin{equation}\label{eq:zmin}
|x|< \frac{1}{2}\min_{k=2,\ldots,K}|x_k|,\end{equation} 
then  
\begin{align}
I_K(x) \label{eq:exp}
&\ge \exp\bigg\{-c_1(4e)^{K}e^{-F(K)} \frac{x^2}{\prod_{k=2}^K|x_k|^2}\bigg\}
\frac{e^{F(K)/2}(4e)^{-K/2}}{\sqrt{2\pi}} |x|\prod_{k=2}^K|x_k|\\
I_K(x) \label{eq:exp1}
&\le \exp\bigg\{-c_2 \Big(\frac{3e}{2}\Big)^{-2K} e^{-F(K)} \frac{x^2}{\prod_{k=2}^K|x_k|^2}\bigg\}
\frac{ e^{F(K)/2} (\frac{3e}{2})^K}{2\sqrt{2\pi}} |x|\prod_{k=2}^K|x_k|
\end{align}
\end{enumerate}
\end{lemma}
\begin{proof}\mbox{}
\begin{enumerate}
\item Since $G$ is strictly positive definite, we have $G(0)>G(x)$ for all $x\ne 0$, and hence there exist constants $c_1',c_2'>0$ such that
\[ c_1'x^2\le G(0)-G(x)\le c_2' x^2,\quad\text{ for all } x\in[-1,1].\]
Using Theorem \ref{thm:2}, identity \eqref{eq:ie3} and Lemma \ref{lm:3}, we then get \eqref{eq:expz},\eqref{eq:expz1} with $c_1=(c_2')^2, c_2=(c_1')^2/2$.  
Note that the $k=1$ factor is not present in the products over $k$ in the exponentials, as it equals $x^2$ and has been cancelled with $x^2$ in the numerator. 
\item From the inequalities
\[\frac{1}{2} |x_k|\le |x_k|-|x|\le |x_k-x|\le |x_k|+|x|\le \frac{3}{2}|x_k|\]
we obtain
\[ \Big(\frac{1}{2}\Big)^{K-1}\prod_{k=2}^K|x_k| 
\le\prod_{k=2}^K|x_k-x|\le 
\Big(\frac{3}{2}\Big)^{K-1}\prod_{k=2}^K|x_k| \]
and then substitute these latter inequalities in \eqref{eq:expz}, \eqref{eq:expz1}.
\end{enumerate}
\end{proof}

\begin{lemma} Under the assumptions of Theorem \ref{thm:3}, for all sufficiently large $K$:  
\begin{enumerate}
\item[a)] 
\begin{equation}\label{eq:ikz0} I_K(x_{K+1}) \ge  e^{2F(K)} \prod_{k=2}^K|x_k|^2,
\end{equation}

\item[b)] 
\begin{equation}\label{eq:zk1}
 |x_{K+1}|\ge  e^{2F(K)}\prod_{k=2}^K|x_k|,
\end{equation}
\item[c)] 
\begin{align}\label{eq:zk2}
 |x_{K+1}| &\ge e^{2^K F(K)},\\
 |I_K(x_{K+1})| &\ge e^{2^K F(K)},\label{eq:ikz1}
\end{align}
\item[d)] 
\begin{equation}\label{eq:zk3}
|x_{K+1}| \le e^{F(K)/3}.
\end{equation}
\end{enumerate}
\end{lemma}
\begin{proof}\mbox{}
\begin{enumerate}
\item[a)] Given $K$, let us choose $x\in[-1,1]$ so  as to make the expression in braces in \eqref{eq:exp} equal to -1, i.e.,
\[ |x| = c_1^{-1/2} (4e)^{-K/2} e^{F(K)/2} \prod_{k=2}^K|x_k|.
\]
By Theorem \ref{thm:2}, $F(K)/K\to-\infty$, so condition \eqref{eq:zmin} holds if $K$ is large enough, and we can apply inequality \eqref{eq:exp}:
\begin{equation}
I_K(x) \ge c_3 (4e)^{-K} e^{F(K)} \prod_{k=2}^K|x_k|^2,
\end{equation}
with some constant $c_3$ depending on $G$.
By definition of $x_{K+1}$, $I_K(x_{K+1})\ge I_K(x)$. Finally, using again $F(K)/K\to-\infty$, we arrive at \eqref{eq:ikz0}.

\item[b)] Suppose that \eqref{eq:zk1} is violated for infinitely many $K\in\N$. 
Then, for sufficiently large such $K$, the value $x_{K+1}$  satisfies  condition \eqref{eq:zmin}  with $x=x_{K+1}$, and we can apply  bound \eqref{eq:exp1}. 
It follows that for such $K$
\[I_K(x_{K+1}) 
\le 
\frac{ e^{5F(K)/2} (\frac{3e}{2})^K}{2\sqrt{2\pi}} \prod_{k=2}^K|x_k|^2 = o(I_K(x_{K+1})),
\] 
where in the last equality we have used \eqref{eq:ikz0} and that $F(K)/K\to-\infty$. Therefore the hypothesis that \eqref{eq:zk1} is violated for infinitely many $K$ is false. 
\item[c)] 
To show \eqref{eq:zk2}, we continue inequality \eqref{eq:zk1} by iteratively applying it to $x_{k+1}$ with $k=K, K-1,\dots, K_0+1$, where $K_0+1$ is the lowest value for which it is valid: 
\begin{align*}
|x_{K+1}| 
&\ge e^{2F(K)}\prod_{k=2}^K|x_k|\\
&\ge e^{2F(K)+2F(K-1)} \prod_{k=2}^{K-1}|x_k|^2\\
&\ge e^{2F(K)+2F(K-1)+4F(K-2)} \prod_{k=2}^{K-2}|x_k|^4\\
&\ldots\\
&\ge e^{2F(K)+\sum_{k=K_0}^{K-1}2^{K-k}F(k)}
\prod_{k=2}^{K_0}|x_k|^{2^{K-K_0}}\\
&\ge e^{2^{K-K_0+1}F(K)} \bigg(\prod_{k=2}^{K_0}|x_k|^{2^{-K_0}}\bigg)^{2^{K}}\\
&= e^{2^KF(K)}\exp\bigg\{2^K\bigg[(2^{-K_0+1}-1)F(K)+\ln \bigg(\prod_{k=2}^{K_0}|x_k|^{2^{-K_0}}\bigg)\bigg]\bigg\},
\end{align*}
where we assumed without loss of generality that $K_0\ge 2$ and that monotonicity of $F(k)$ established in Theorem \ref{thm:2} holds for $k\ge K_0$. The second exponential factor in the last expression is greater than 1 for sufficiently large $K$ due to $F(K)\to-\infty$, which implies \eqref{eq:zk2}.

Inequality  \eqref{eq:ikz1} is proved in the same way, using \eqref{eq:ikz0} instead of \eqref{eq:zk1} in the first step.
\item[d)]   Suppose that \eqref{eq:zk3} is violated for infinitely many $K\in\N$. 
First, observe that for sufficiently large such $K$ the bound \eqref{eq:expz1}, when applied to $x=x_{K+1}$, implies
\begin{equation}\label{eq:ef3}
I_K(x_{K+1})\le \exp\{-e^{-F(K)/4}\}.
\end{equation}
Indeed, consider the first, exponential factor in \eqref{eq:expz1}. 
Using $|x_{K+1}| > e^{F(K)/3},$   the inequalities $|x_k-x_{K+1}|\le 2$, and $F(K)/K\to -\infty$, we can write:
\[-c_2e^{-2K-F(K)} \frac{x_{K+1}^2}{\prod_{k=2}^K|x_k-x_{K+1}|^2}
\le -c_2 (2e)^{-2K} e^{-F(K)/3} \le -e^{-F(K)/4}\]
for $K$ large enough. As for the remaining factor,  
\[\frac{e^{F(K)/2+K}}{\sqrt{2\pi}} |x_{K+1}|\prod_{k=2}^K|x_k-x_{K+1}|,\]
it is bounded by 1 for large $K$, again due to $F(K)/K\to -\infty$. 
We thus conclude \eqref{eq:ef3}.

Now, combining \eqref{eq:ef3} with \eqref{eq:ikz1}, we see that 
\[
e^{2^K F(K)} \le I_K(x_{K+1}) \le \exp\{-e^{-F(K)/4}\}.
\]
This implies $2^K (-F(K))\ge e^{-F(K)/4}$. Since $c\le e^{c/8}$ for sufficiently large $c,$ we get $2^K\ge  e^{-F(K)/8}$.
But this inequality is violated for all $K$ large enough, since $F(K)/K\to -\infty$.
Therefore our assumption that  \eqref{eq:zk3} is violated for infinitely many $K\in\N$ was wrong.  
\end{enumerate}
\end{proof}
Inequalities \eqref{eq:zk2} and \eqref{eq:zk3} form the statement \eqref{eq:thm3} of Theorem \ref{thm:3}. 
Since $F(K)/K\to -\infty$, from \eqref{eq:zk3} we conclude $|x_{K}|\to 0$, which completes the proof.

\section*{Acknowledgement}
The author thanks the two anonymous referees for the careful reading of the manuscript and several valuable suggestions and corrections.


\begin{thebibliography}{10}

\bibitem{bull}
A.~D. Bull.
\newblock Convergence rates of efficient global optimization algorithms
(2011).
\newblock {\em Journal of Machine Learning Research}, 12:2879-2904 (2011).

\bibitem{forrester}
A. I. J. Forrester, A. S{\'o}bester, and A. J. Keane.
\newblock {\em Engineering design via surrogate modelling: a practical guide.}
\newblock J. Wiley (2008).


\bibitem{Gutmann:2001:RBF:596093.596381}
H.-M. Gutmann.
\newblock A radial basis function method for global optimization.
\newblock {\em J. of Global Optimization}, 19:201--227  (2001).

\bibitem{Jones:2001:TGO:596097.596412}
D.~R. Jones.
\newblock A taxonomy of global optimization methods based on response surfaces.
\newblock {\em J. of Global Optimization}, 21:345--383 (2001).

\bibitem{JoScWe97a}
D.~R. Jones, M.~Schonlau, and W.~J. Welch.
\newblock {A data analytic approach to Bayesian global optimization}.
\newblock In {\em Proceedings of the ASA, Section on Physical and Engineering
  Sciences},  186 -- 191 (1997).

\bibitem{Jones:1998:EGO:596070.596218}
D.~R. Jones, M.~Schonlau, and W.~J. Welch.
\newblock Efficient global optimization of expensive black-box functions.
\newblock {\em J. of Global Optimization}, 13:455--492,  (1998).

\bibitem{katznelson2004introduction}
Y.~Katznelson.
\newblock {\em An introduction to harmonic analysis}.
\newblock Cambridge mathematical library. Cambridge University Press (2004).

\bibitem{kincaid2002numerical}
D.~Kincaid and E.~Cheney.
\newblock {\em Numerical analysis: mathematics of scientific computing}.
\newblock Pure and applied undergraduate texts. American Mathematical Society (2002).

\bibitem{Locatelli:1997:BAO:596056.596117}
M.~Locatelli.
\newblock Bayesian algorithms for one-dimensional global optimization.
\newblock {\em J. of Global Optimization}, 10:57--76 (1997).

\bibitem{mockus1989bayesian}
J.~Mockus.
\newblock {\em Bayesian approach to global optimization: theory and
  applications}.
\newblock Mathematics and its applications: Soviet series. Kluwer Academic (1989).

\bibitem{Mockus_Tiesis_Zilinskas_1978}
J.~Mockus, V.~Tiesis, and A.~Zilinskas.
\newblock The application of bayesian methods for seeking the extremum.
\newblock {\em Towards Global Optimization}, 2:117 -- 129 (1978).

\bibitem{mpmath}
\textsc{Mpmath} library:
\newblock{\url{http://code.google.com/p/mpmath/}}

\bibitem{rasmussen}
C. E. Rasmussen and C. K. I. Williams
\newblock {\em Gaussian Processes for Machine Learning}.
\newblock The MIT Press (2006).

\bibitem{Schonlau:1997:CEG:926099}
M.~Schonlau.
\newblock {\em Computer experiments and global optimization}.
\newblock PhD thesis, Waterloo, Ont., Canada (1997).

\bibitem{schonlau1}
M.~Schonlau and W.~J. Welch.
\newblock Global optimization with nonparametric function fitting.
\newblock {\em Proceedings of the ASA, Section on Physical and Engineering
  Sciences}, 183 -- 186 (1996).

\bibitem{Torn:1989:GO:75021}
A.~Torn and A.~Zilinskas.
\newblock {\em Global optimization}.
\newblock Springer-Verlag New York (1989).

\bibitem{VAZQUEZ:2010:HAL-00440827:1}
E.~Vazquez and J.~Bect.
\newblock {Pointwise consistency of the kriging predictor with known mean and   covariance functions}.
\newblock In {\em {mODa 9 -- Advances in Model-Oriented Design and Analysis}}.
\newblock 14th-19th June 2010, Bertinoro, Italy.

\bibitem{Vazquez20103088}
E.~Vazquez and J.~Bect.
\newblock Convergence properties of the expected improvement algorithm with
  fixed mean and covariance functions.
\newblock {\em Journal of Statistical Planning and Inference}, 140(11):3088 --
  3095 (2010).

\end{thebibliography}
\end{document}